\documentclass[11pt,oneside,leqno]{amsart}
\usepackage{amsxtra}
\usepackage{amsopn}
\usepackage{amsmath,amsthm,amssymb}
\usepackage{color}
\usepackage{amscd}
\usepackage{amsfonts}
\usepackage{latexsym}
\usepackage{verbatim}
\usepackage{pb-diagram}
\usepackage{pifont}
\usepackage{graphicx}

\theoremstyle{plain}
\newtheorem{theorem}{Theorem}[section]
\newtheorem*{theorem*}{Theorem}

\newtheorem{prop}[theorem]{Proposition}

\newtheorem*{mt*}{Main Theorem}
\newcommand\g{{\mathfrak{g}}}
\newcommand\h{{\mathfrak{h}}}

\newcommand\su{{\mathfrak{su}}}

\newcommand\m{{\mathfrak{m}}}

\renewcommand\t{{\theta}}

\newcommand{\D}{\mathcal D}
\newcommand{\A}{\mathcal A}

\begin{document}
\title[Some geometrical properties of the Oscillator group]{Some geometrical properties of the Oscillator group}
\author{Y. Aryanejad}
\date{}

\address{Department of Mathematics\\Payame noor University\\P.O. Box 19395-3697\\Tehran\\Iran.}
\email{Yadollah Aryanejad: y.aryanejad@pnu.ac.ir}

\begin{abstract}
We consider the oscillator group equipped with
a biinvariant Lorentzian metric, and then
 some geometrical properties of this group i.e. homogeneous
Ricci solitons and harmonicity properties of invariant vector fields are obtained. We also determine all vector fields which are critical points for the energy functional
restricted to vector fields. Vector fields defining harmonic maps are also determined, and the energy of critical points is explicitly calculated.
\end{abstract}

\maketitle

\section{Introduction}
The oscillator group is defined as the semidirect product of the line (time) with the Heisenberg group, with the action given by the dynamics. Thus, it is not realised as a matrix group. It is a solvable group, but not an exponential group. The oscillator groups, called Warped Heisenberg Lie groups in \cite{Z1}, are not only important in Lorentzian geometry
but also have interesting applications in Conformal Field Theory, in WZWmodels
(see \cite{N1}) and Supergravity.
This group has many properties useful both in geometry and physics. To mention but two
geometrical applications, Medina \cite{Mn} proved that Os is the only four-dimensional non-Abelian simply connected
solvable Lie group, which admits a bi-invariant Lorentzian metric; and Console, Ovando,
and Subils \cite{Cn} obtain solvable models for Kodaira surfaces by using suitable lattices on the oscillator
group.
 Moreover, it is \cite{Ln} an example of
homogeneous spacetime, which as causal space satisfies the so-called causal continuity.
Levichev studied in \cite{Ln1} the oscillator group with the biinvariant Lorentzian metric,
which geometrically is a Lorentzian
symmetric space and phisically is related to an isotropic electromagnetic field.\\
The oscillator group has interesting features from the viewpoints of both Differential Geometry
and Physics (see, for example, \cite{Cn1}, \cite{Dn} and the references therein). 
On the other hand, In paper \cite{on}, Onda has constructed Lorentzian algebraic Ricci solitons on the oscillator groups $G_m(\lambda)$. In particular, he obtained new Lorentzian Ricci solitons $G_m(\lambda)$ which in compare with our result is wrong.\\
 A natural generalization of an Einstein manifold is Ricci soliton, i.e.   
a pseudo Riemannian metric $g$ on a smooth manifold $M$, such that the equation
\begin{eqnarray}\label{ric}
\begin{array}{cccc}
\mathcal{L}_{X} g =\varsigma g-\varrho,
\end{array}
\end{eqnarray}
holds for some $\varsigma \in R$ and some smooth vector field $X$ on $M$, where $\varrho$
denotes the Ricci tensor of $(M, g)$ and $\mathcal{L}_{X}$ is the usual Lie derivative.
 According
to whether $\varsigma > 0, \varsigma = 0$ or $\varsigma < 0$ a Ricci soliton $g$ is said to be shrinking, steady or expanding, respectively. A homogeneous Ricci soliton on a homogeneous space $M = G/H$ is a G-invariant metric $g$ for which the equation \eqref{ric} holds and
an invariant Ricci soliton is a homogeneous apace, such that equation \eqref{ric} is satisfied by an invariant vector field.
Indeed, the study of Ricci solitons homogeneous spaces is an interesting area
of research in pseudo-Riemannian geometry.
For example, evolution of homogeneous Ricci solitons under the bracket flow \cite{Lm}, algebraic solitons and the Alekseevskii Conjecture properties\cite{LM}, conformally flat Lorentzian gradient Ricci
solitons\cite{MB}, properties of algebraic Ricci Solitons of three-dimensional Lorentzian
Lie groups\cite{BA}, algebraic Ricci solitons \cite{Ba}. Non-K\"{a}hler examples of Ricci solitons are very hard to find yet (see \cite{DH}).
In case $(G, g)$ be a simply-connected completely solvable Lie
group equipped with a left-invariant metric, and $(\g,\langle,\rangle )$ be the corresponding
metric Lie algebra, then $(G, g)$ is a Ricci soliton if and only if $(\g,\langle ,\rangle)$ is an
algebraic Ricci soliton \cite{LJ}. \\
Up to our knowledge, no geometrical properties such as harmonicity properties of invariant vector fields have been obtained yet for the oscillator group.
Investigating critical points of the energy associated to vector fields is an interesting purpose
under different points of view. As an example by the
Reeb vector field $\xi$ of a contact metric manifold, somebody can see how the criticality of such a vector field is related to the geometry of the manifold (\cite{p1},\cite{p2}).
Recently, it has been \cite{g3} proved that critical points of $E:\mathfrak{X}(M)\rightarrow R$, that is,
the energy functional restricted to vector fields, are again parallel vector fields. Moreover,
in the same paper it also has been determined the tension field associated to a unit vector field $V$, and investigated the problem of determining when $V$ defines a harmonic map. \\
A Riemannian manifold admitting a parallel vector field is locally reducible, and the same is true for a
pseudo-Riemannian manifold admitting an either space-like or time-like parallel vector field. This leads us to consider different situations, where some interesting types of non-parallel vector fields can be characterized in terms of harmonicity properties. We may refer
to the recent monograph \cite{d1} and some references \cite{i}, \cite{n} for an overview on harmonic vector fields.\\
As for the contents, in Section 2,we give some preliminaries.
In Section 3, we investigate rquired conditions for oscillator group Ricci solitons. Harmonicity properties
of vector fields on oscillator group will be determined
in Sections 4. Finally, the energy of all these vector fields is explicitly calculated in Section 5.
\section{preliminaries}
Let $M = G/H$ be a homogeneous manifold (with $H$ connected), $\g$ the Lie algebra
 of $G$ and $\h$ the isotropy subalgebra. Consider $\m =\g /\h$ the factor space, which identies
with a subspace of $\g$ complementary to $\h$. The pair $(\g,\h)$ uniquely defines the isotropy
representation
\begin{center}
$\psi :\g \longrightarrow \mathfrak{gl}(\m),\quad      \psi(x)(y)=[x,y]_\m$   
\end{center}
for all  $x\in \g, y\in \m$. Suppose that $\lbrace e_1,...,e_r,u_1,...,u_n\rbrace$ be a basis of $\g$, where $\lbrace e_j\rbrace$ and $\lbrace u_i\rbrace$ are bases of $\h$ and $\m$
respectively, then with respect to $\lbrace u_i \rbrace$, $H_j$ whould be the isotropy representation  for $e_j$.
$g$ on $\m$ uniquely defines its invariant linear Levi-Civita connection, as the corresponding homomorphism of $\h$-modules  $\Lambda:\g \longrightarrow \mathfrak{gl}(\m)$ such that $\Lambda(x)(y_\m)=[x,y]_\m$ for all  $x\in \h, y\in \g$. In other word
\begin{eqnarray}\label{con}
\begin{array}{cccc}
\Lambda(x)(y_\m)=\frac{1}{2}[x,y]_\m+v(x,y)
\end{array}
\end{eqnarray}
for all  $x,y\in \g$, where $v:\g \times \g\rightarrow \m$ is the $\h$-invariant symmetric mapping uniquely determined by
\begin{center}
$2g(v(x, y), z_\m) = g(x_\m, [z,y]_\m) + g(y_\m,[z,x]_\m)$
\end{center}
for all  $x,y,z\in \g$,
Then the curvature tensor can be determined by
\begin{eqnarray}\label{curve}
R:\m \times \m \longrightarrow \mathfrak{gl}(\m),\quad & R(x,y)=[\Lambda(x),\Lambda(y)]-\Lambda ([x,y]),   
\end{eqnarray}
and with respect to $u_i$, the Ricci tensor $\rho$ of $g$  is given by
\begin{eqnarray}\label{ric2}
\rho (u_i ,u_j)=\sum_{k=1}^4g(R(u_k,u_i)u_j,u_k),\quad i,j=1,\dots,4.
\end{eqnarray}
Let $(M,g)$ be a compact Riemannian manifold and $g_s$ be the Sasaki metric
on the tangent bundle $TM$, then the energy of a smooth vector field $V:(M,g)\longrightarrow
(TM,g^s)$ on $M$ is;
\begin{equation}\label{enr}
E(V)=\dfrac{n}{2}vol (M,g)+\dfrac{1}{2}\int_M ||\nabla V||^2dv
\end{equation}
(assuming $M$ compact; in the non-compact case, one works over relatively compact domains see \cite{c1}). If $V:(M,g)\longrightarrow (TM,g^s)$ be a critical point for the energy functional, then $V$ is said to define a harmonic map. The Euler-Lagrange equations characterize vector fields $V$ defining harmonic maps as the ones whose tension field $\t(V)=tr(\nabla^2V)$ vanishes.
 Consequently,
$V$ defines a harmonic map from $(M,g)$ to $(TM,g^s)$ if and only if
\begin{equation}\label{hor}
 tr[R(\nabla_. V,V)_.]=0, \quad  \nabla^*\nabla V=0,
\end{equation}
where with respect to an orthonormal local frame $\lbrace e_1,...,e_n\rbrace$ on $(M,g)$, with $\varepsilon_i=g(e_i,e_i)=\pm1$ for all indices i,
one has
\begin{center}
$ \nabla^*\nabla V=\sum_i \varepsilon_i( \nabla_{e_i}\nabla_{e_i} V-\nabla_{\nabla_{e_i}e_i}V)$.   
\end{center}
A smooth vector field V is said to be a harmonic
section if and only if it is a critical point of $E^v(V)=(1/2)\int_M||\nabla V||^2dv$ where $E^v$ is the vertical energy. The corresponding Euler-Lagrange equations are given by
\begin{equation}
 \nabla^*\nabla V=0,
\end{equation}
Let $\mathfrak{X}^{\rho}(M) =\lbrace V\in \mathfrak{X}(M): ||V||^2=\rho^2 \rbrace$ and $\rho\neq 0$. Then, one can consider vector fields $ V\in \mathfrak{X}(M)$
which are critical points for the energy functional $E
|_{\mathfrak{X}^{\rho}(M)}$, restricted to vector fields of the same constant length. The
Euler-Lagrange equations of this variational condition are given by
\begin{equation}\label{hor1}
\nabla^*\nabla V\quad is\quad collinear\quad to\quad V.   
\end{equation}
As usual, condition \eqref{hor1} is taken as a definition
of critical points for the energy functional restricted to vector fields of the same length in the non-compact case.

\section{Homogeneous Ricci solitons on oscillator group}
We consider the basis $\lbrace P,X_1,Y_1,Q \rbrace$ of the Lie
algebra $\g$ with brackets
\begin{equation}\label{lie8}
[X_1,Y_1]=P,\quad [Q,X_1]=Y_1,\quad   [Q,Y_1]=-X_1.
 \end{equation}
 The corresponding simply connected Lie group $G$ is called the oscillator group.
Consider the biinvariant Lorentzian metric $g$ on the oscillator group $G$ given in the basis $\lbrace P,X_1,Y_1,Q \rbrace$, by
\begin{equation}\label{meter}
g= \left( \begin{array}{cccc}
   0 & 0 & 0 & 1   \\
   0 & 1 & 0  & 0 \\
   0 & 0 & 1  &0 \\
    1 & 0 & 0  &0
   \end{array}  \right).
\end{equation}
The components of the Levi-Civita connection are calculated using the well known {\em Koszul} formula and are
\begin{equation}
\begin{array}{cccc}
 \Lambda_1=0, &\quad \quad
 \Lambda_2= \left( \begin{array}{cccc}
  0 & 0 & \frac{1}{2} & 0   \\
   0 & 0 & 0  & 0 \\
   0 & 0 & 0  & - \frac{1}{2} \\
    0 & 0 & 0  &0
   \end{array}  \right), &\\
   
    \Lambda_3= \left( \begin{array}{cccc}
    0 & 0 & -\frac{1}{2} & 0   \\
   0 & 0 & 0  & 0 \\
   0 & 0 & 0  & \frac{1}{2} \\
    0 & 0 & 0  &0 
   \end{array}  \right), &\quad \quad 
   \Lambda_4= \left( \begin{array}{cccc}
   0 & 0 & 0 & 0   \\
   0 & 0 &  -\frac{1}{2} & 0 \\
   0 & \frac{1}{2} & 0  &  0 \\
    0 & 0 & 0  &0
   \end{array}  \right).
   \end{array}
 \end{equation}
Using \eqref{curve} we can determine the non-zero curvature components;
$$
 \begin{array}{cc}
   R(X_1,Q)X_1=\frac{1}{4} P ,& \quad R(Y_1,Q)Y_1=\frac{1}{4} P, \\
  R(Q,X_1)Q=\frac{1}{4} X_1, & \quad R(Q,Y_1)Q=\frac{1}{4} Y_1.
     \end{array} 
    $$
Since $R(X, Y,Z,W) = g(R(X, Y )Z,W)$ we have;
\begin{equation}\label{ric33}
 \begin{array}{cc}
   R(X_1,Q,X_1, Q)= R(Q,Y_1,Q, Y_1)=\frac{1}{4}.
     \end{array} 
\end{equation}
Applying the Ricci tensor formula \eqref{ric2}, we get;
\begin{equation}\label{ric3}
(\rho_{ij})= \left( \begin{array}{cccc}
   0 & 0 & 0 & 0   \\
   0 & 0 & 0 & 0 \\
   0 & 0 & 0  &  0 \\
    0 & 0 & 0  & \frac{1}{2}
    \end{array}  \right).
\end{equation}
which is diagonal with eigenvalue $r_1=\frac{1}{2}$. For an arbitrary left-invariant vector field $X =aP+bX_1+cY_1+dQ\in \g$ we have;
$$
 \begin{array}{cccc}
  \nabla_{P}X=0 ,& \nabla_{X_1}X= \frac{1}{2}cP- \frac{1}{2}d Y_1,& \\
   \nabla_{ Y_1}X= -\frac{1}{2}bP+ \frac{1}{2}dX_1,&
   \nabla_{Q}X= -\frac{1}{2}cX_1+\frac{1}{2}b Y_1. 
     \end{array}
$$
Although once the metric is fixed and bi-invariant, for an arbitrary left-invariant vector field $ \mathcal{L}_{X}g=0$, using the relation $(\mathcal{L}_{X}g)(Y,Z) = g(\nabla_Y X,Z) + g(Y,\nabla_ZX)$ one concludes that;
 \begin{equation}\label{kil}
  \mathcal{L}_{X}g=0
\end{equation}
So by the Ricci soliton formula \eqref{ric}, we get a system of differential equations including $ \frac{1}{2}=0$ which means that the system of differential equations is incompatible.
Thus, we have the following.
\begin{theorem} 
Let $G$ be  the oscillator group equipped with
biinvariant Lorentzian metric $g$ described in \eqref{meter}, then $G$ can not be a Ricci soliton manifold.
\end{theorem}
As we already mentioned by theorem 2.5 in \cite{on} this result contradicts theorem 4.1 in \cite{on}. An algebraic Ricci soliton in this sense is a pseudo-Riemannian 
$(M,g)$ metric
satisfying
\begin{equation}\label{algric}
Rc = cId + D
\end{equation}
where $Rc$ denotes the Ricci operator, $c$ is a real number, and $D \in Der(\g)$. Under
the hypothesis of theorem 4.1 in \cite{on}, set $m=\lambda_1=1$ and $\epsilon=0$. So, we have lie algebra $\g$ and lorentzian metric $g$ described in \eqref{lie8} and \eqref{meter} repectively. By formula 3.3 in \cite{on}, the derivation $D$ must be as following
$$
D= \left( \begin{array}{cccc}
   D_1^1 & D_2^1 & D_3^1 & D_4^1   \\
   D_1^2 & D_2^2 & D_3^2 & D_4^2 \\
   D_1^3 & D_2^3 & D_3^3 & D_4^3 \\
    D_1^4 & D_2^4 & D_3^4 & D_4^4
     \end{array}  \right),
$$
where by Eq(4.5) in \cite{on}, $DQ= \mu P=\frac{1}{2} P$ and hence $D_1^4=\mu=\frac{1}{2}$. It means that 
\begin{equation}\label{D}
D= \left( \begin{array}{cccc}
   0 & 0 & 0 & 0   \\
   0 & 0 & 0 & 0 \\
   0 & 0 & 0  &  0 \\
    \frac{1}{2} & 0 & 0  & 0
    \end{array}  \right).
\end{equation}
Also using Eq(4.5) in \cite{on}, the Ricci operator $Rc$ is given by
\begin{equation}\label{Rc}
Rc= \left( \begin{array}{cccc}
   0 & 0 & 0 & \frac{1}{2}   \\
   0 & 0 & 0 & 0 \\
   0 & 0 & 0  &  0 \\
    0 & 0 & 0  & 0
    \end{array}  \right).
\end{equation}    
As we can see $Rc\neq D$, therefore $DQ= \mu P=\frac{1}{2} P$ is wrong and theorem 4.1 in \cite{on} can not be true.
By \eqref{ric3}, $\rho_{ij}\neq\lambda g_{ij}$ for all indices i, j, so, we proved the following result too.
\begin{prop} 
Let $G$ be  the oscillator group equipped with
biinvariant Lorentzian metric $g$ described in \eqref{meter}, then $G$ can not be an Einstein manifold.
\end{prop}
We denote the Ricci operator and the scalar curvature by $Rc$ and $\tau$ respectively. Let $M_q^n$ be a pseudo-Riemannian manifold of index $q$. The Weyl conformal curvature
tensor field $C$ of type $(1,3)$ of $M$ is defined by 
\begin{equation}\label{CurvFromRic}
\begin{array}{cr}
C(X,Y)Z=R(X,Y)Z-(\frac{1}{n-2}(QX\wedge Y+X\wedge QY)Z
+\frac{\tau}{(n-1)(n-2)}
(X\wedge Y )Z,
\end{array}
\end{equation}
where $(X\wedge Y )Z=<Y,Z>X-<X,Z>Y$.
It is well-known \cite{BA1} that
for a conformally
flat space the curvature tensor can be completely determined using the Ricci tensor.
Moreover, if $n\geqslant 4$, then $M_q^n$ is conformally 
flat if and only if $C=0$.
\begin{prop} 
The oscillator group equipped with
biinvariant Lorentzian metric $g$ described in \eqref{meter} is conformally flat.
\end{prop}
\begin{proof}
Since the scalar curvature is $\tau=\sum_i(\rho_i,\rho_i)$ (see \cite{be1}. p. 43), by \eqref{ric3}, $\tau=\frac{1}{2}$.
Using \eqref{CurvFromRic} and \eqref{Rc} a straightforward calculation then yields that $C=0$, as desired.
\end{proof}
A D' Atri space is defined as a Riemannian manifold $(M, g)$ whose local geodesic symmetries are volumepreserving. Let us recall that the property of being a D' Atri space is
equivalent to the infinite number of curvature identities called the odd Ledger conditions $L_{2k+1}$, $k\geq 1$ (see \cite{d11}, \cite{s1}). In particular, the two first non-trivial Ledger conditions are:
\begin{equation}
L_3: (\nabla_X \rho)(X,X)=0,\quad L_5: \sum_{a,b=1}^nR(X,E_a,X,E_b)(\nabla_X R)(X,E_a,X,E_b)=0,
\end{equation}
where $X$ is any tangent vector at any point $m\in M$ and $\lbrace E_1, . . . , E_n\rbrace$ is any orthonormal basis of $T_mM$. Here $R$ denotes
the curvature tensor and $\rho$ the Ricci tensor of $(M, g)$, respectively, and $n = dimM$.\\
Thus, it is natural to start with the investigation of the oscillator group satisfying the simplest Ledger condition $L_3$, which is the first approximation of the D' Atri property. This condition is called in \cite{p112} "the class
$\A$ condition". Equivalently Ledger condition $L_3$ holds if and only if
the Ricci tensor is cyclic-parallel, i.e.
\begin{center}
 $ (\nabla_X \rho)(Y,Z)+ (\nabla_Y \rho)(Z,X)+ (\nabla_Z \rho)(X,Y)=0$.
\end{center}
For more detail see \cite{Ca1}.
\begin{prop} 
The oscillator group equipped with
biinvariant Lorentzian metric described in \eqref{meter} is a D' Atri space which its first approximation holds.
\end{prop}
\begin{proof}
In Ledger condition $L_3$,
\begin{center}
$\nabla_i\rho_{jk}=-\sum_t(\varepsilon_jB_{ijt}\rho_{tk}+\varepsilon_k B_{ikt}\rho_{tj}),$
\end{center}
where  $B_{ijk}$ components can be obtained by the relation $\nabla_{e_i}e_j =\sum_k\varepsilon_j B_{ijk}e_k$ with $\varepsilon_i=g(e_i,e_i)=\pm1$ for all indices i.
But $\nabla_1\rho_{11}=\nabla_2\rho_{22}=\nabla_3\rho_{33}=0$, hence the Ricci tensor is cyclic-parallel and the first approximation of the D' Atri property holds.
\end{proof}
A pseudo-Riemannian manifold which admits a parallel degenerate distribution is called a {\em Walker} manifold. Walker spaces were introduced by Arthur Geoffrey Walker in 1949. The existence of such structures causes many interesting properties for the manifold with no Riemannian counterpart. Walker also determined a standard local coordinates for these kind of manifolds \cite{Wa1,Wa2}.
\begin{prop} 
Let $G$ be the oscillator group equipped with
a biinvariant Lorentzian metric $g$ described in \eqref{meter}, then $(G, g)$ admits invariant parallel degenerate line field $\D$ with the generator $\lbrace P \rbrace$.
\end{prop}
\begin{proof}
Set $X =aP+bX_1+cY_1+dQ\in \g$ and suppose that $\D={\rm span}(X)$ is an invariant null parallel line field. Then, the following equations must satisfy for some parameters $\omega_1,\dots,\omega_4$
$$
\begin{array}{llll}
\nabla_{P}X=\omega_1X,&\nabla_{X_1}X=\omega_2X,&\nabla_{Y_1}X=\omega_3X,&\nabla_{Q}X=\omega_4X.
\end{array}
$$
By straight forward calculations we conclude that the following equations must satisfy
$$
\begin{array}{l}
\omega_1a=0,\quad \omega_1b=0,\quad\omega_1c=0,\quad \omega_1d=0,\\
\omega_2b=0,\quad \omega_2d=0,\quad
 -\omega_2a+\frac{1}{2}c=0,\quad -\omega_2c+\frac{1}{2}d=0,\\
\omega_3c=0,\quad \omega_3d=0,\quad
 -\omega_3a-\frac{1}{2}b=0,\quad -\omega_3b+\frac{1}{2}d=0,\\
\omega_4a=0,\quad \omega_4d=0,\quad
 -\omega_4b-\frac{1}{2}c=0,\quad -\omega_4c+\frac{1}{2}b=0.
\end{array}
$$
$X$ is null, hence $X$ must satisfy $g(X,X)=2ad+b^2+c^2=0$ described in \eqref{meter}. 
By solving the above system of equations we obtain that $X=a P$. It means that $b=c=d=0$. 
\end{proof}

\section{Harmonicity of vector fields on oscillator group}
In this section we investigate the harmonicity of invariant vector fields on the oscillator group equipped with biinvariant Lorentzian metric $g$ described in \eqref{meter}.\\
We can construct an orthonormal frame field $\lbrace e_1,e_2,e_3,e_4\rbrace$ with respect to $g$;
\begin{center}
$e_1=-P+X_1,\quad e_2=X_1+Q,\quad e_3=Y_1,\quad e_4=-P+X_1+Q,$
\end{center}
with $e_1,e_2,e_3$ space-like and $e4$ time-like. We get;
\begin{eqnarray}
\begin{array}{ccc}
 [e_1,e_2]=-e_3,&\quad [e_1,e_3]=e_2-e_4,&\quad [e_1,e_4]=-e_3,\\
 
 [e_2,e_3]=-e_1,&\quad  [e_3,e_4]=e_1.&
 \end{array}
 \end{eqnarray}
Considering formula \eqref{con} the connection components are;
\begin{eqnarray}\label{con1}
\begin{array}{ccc}
\nabla_{e_1}e_2=-\frac{1}{2}e_3,&\quad \nabla_{e_1}e_3=\frac{1}{2}e_2-\frac{1}{2}e_4,&\quad \nabla_{e_1}e_4=-\frac{1}{2}e_3,\\
\nabla_{e_2}e_1=-\frac{1}{2}e_3,&\quad \nabla_{e_2}e_3=-\frac{1}{2}e_1,&\\
\nabla_{e_3}e_1=-\frac{1}{2}e_2+\frac{1}{2}e_4,&\quad \nabla_{e_3}e_2=\frac{1}{2}e_1,&\quad \nabla_{e_3}e_4=\frac{1}{2}e_1,\\
\nabla_{e_4}e_1=\frac{1}{2}e_3,&\quad \nabla_{e_4}e_3=-\frac{1}{2}e_1.&
\end{array}
\end{eqnarray}
while $\nabla_{e_i}e_j=0$ in the remaining cases.\\
Set $u=e_2-e_4$. Then, from \eqref{con1} we get $\nabla_{e_i}u=0$ for all indices i. Therefore, $u$ is a parallel light-like vector field. The existence of a light-like parallel vector field is an interesting phenomenon which has no Riemannian counterpart, and characterizes a class of pseudo-Riemannian manifolds which illustrate many of differences between Riemannian and pseudo-Riemannian settings (see for example \cite{ch1},\cite{ch2}).\\
For an arbitrary left-invariant vector field $V =ae_1+be_2+ce_3+de_4 \in \g$ we can now use \eqref{con1}  to calculate $\nabla_{e_i}V$ for all indices i. We get
\begin{eqnarray}\label{con2}
\begin{array}{cc}
\nabla_{e_1}V=-\frac{1}{2}(b+d)e_3+\frac{1}{2}cu,&\quad \nabla_{e_2}V=-\frac{1}{2}ce_1+\frac{1}{2}ae_3,\\
\nabla_{e_3}V=\frac{1}{2}(b+d)e_1-\frac{1}{2}au,&\quad \nabla_{e_4}V=-\frac{1}{2}ce_1+\frac{1}{2}ae_3.
\end{array}
\end{eqnarray}
where the special role of $u=e_2-e_4$ is clear. 
We can now calculate $\nabla_{e_i}\nabla_{e_i}V$  for all indices i. We obtain
\begin{eqnarray}
\begin{array}{cc}\label{con3}
\nabla_{e_1}\nabla_{e_1}V=-\frac{1}{4}(b+d)u,& \nabla_{e_2}\nabla_{e_2}V=-\frac{1}{4}(ae_1+ce_3),\\
\nabla_{e_3}\nabla_{e_3}V=-\frac{1}{4}(b+d)u,& \nabla_{e_4}\nabla_{e_4}V=-\frac{1}{4}(ae_1+ce_3).
\end{array}
\end{eqnarray}
And for $\nabla_{\nabla_{e_i}e_i}V$ for all indices i
\begin{eqnarray}
\begin{array}{cr}\label{con3}
\nabla_{\nabla_{e_1}e_1}V=\nabla_{\nabla_{e_2}e_2}V=\nabla_{\nabla_{e_3}e_3}V=
\nabla_{\nabla_{e_4}e_4}V=0.
\end{array}
\end{eqnarray}
Thus, we find
\begin{center}
$ \nabla^*\nabla V=\sum_i \varepsilon_i( \nabla_{e_i}\nabla_{e_i} V-\nabla_{\nabla_{e_i}e_i}V)=-\frac{1}{2}(b+d)u$.   
\end{center}
If $b=-d$ then $\nabla^*\nabla V=0$.
In the other direction, let $V=ae_1+bu+ce_3 \in \g$. A direct calculation
yields that $\nabla^*\nabla V=0$.

Now, using \eqref{curve} and  \eqref{con3},
we find 
\begin{center} 
$R(\nabla_{e_1} V,V)e_1=0,\quad  R(\nabla_{e_2} V,V)e_2=\dfrac{1}{8}(b+d)(ce_1-ae_3),$\\
$R(\nabla_{e_3} V,V)e_3=0,\quad  R(\nabla_{e_4} V,V)e_4=\dfrac{1}{8}(b+d)(ce_1-ae_3).$
\end{center}
Therefore
\begin{center} 
$tr[R(\nabla_. V,V)_.]=\sum_{i} \varepsilon_i R(\nabla_{e_i} V,V)e_i=0, $
\end{center}
with $\varepsilon_i=g(e_i,e_i)=\pm1$ for all indices i.
Thus, we have the following.
\begin{theorem}\label{hor4}
Let $G$ be the oscillator group equipped with biinvariant Lorentzian metric $g$ described in \eqref{meter} and  $V=ae_1+be_2+ce_3+de_4\in \g$ be a left-invariant vector field on $G$ for some real constants $a,b,c,d$, then   the
following conditions are equivalent:
 \begin{itemize}
\item[$(1)$] $V$ defines a harmonic map;
\item[$(2)$] $V$ is harmonic;
\item[$(3)$] $V$ is a critical point for the energy functional restricted to vector fields of the same length;
\item[$(4)$] $V=ae_1+bu+ce_3$, that is, $b=-d$.
\end{itemize}
\end{theorem}
Therefore, left-invariant vector fields defining a harmonic map form a three-parameter family. As $||ae_1+bu+ce_3||^2=a^2+c^2$ such vector fields are either space-like or light-like.\\
A vector field $V$ is geodesic if $\nabla_VV =0$, and is Killing if $\mathcal{L}_V g=0$, where $\mathcal{L}$ denotes
the Lie derivative. Parallel vector fields are both geodesic and Killing, and vector fields with these special geometric features often have particular harmonicity properties \cite{a2,g11,g2,h1}. By standard calculations we obtain the following result.
\begin{prop} \label{pro1} 
Let $G$ be the oscillator group equipped with biinvariant Lorentzian metric $g$ described in \eqref{meter} and  $V\in \g$ be a left-invariant vector field on $G$, then $V$ is geodesic. Moreover, using \eqref{kil}, we see that $V$ is Killing too.
\end{prop}
Also, with regard to harmonicity properties of invariant vector fields, the oscillator groups display
some particular features. The main geometrical reasons for the special behaviour
of these groups are the existence of a parallel light-like vector field.\\
 Using Proposition \ref{pro1} and Theorem \ref{hor4} a straightforward calculation proves the following classification result, which emphasizes once again the special role played
by the parallel vector field $u$.
\begin{theorem}\label{hor4}
Let $V=ae_1+be_2+ce_3+de_4\in \g$ be a left-invariant vector field on the oscillator group, then   the
following conditions are equivalent:
 \begin{itemize}
\item[$(1)$] $V$ is geodesic;
\item[$(2)$] $V$ is Killing;
\item[$(3)$] $V$ is parallel if and only if $a=c=b-d=0$, that is, $V$ is collinear to $u$.
\end{itemize}
\end{theorem}
\section{The energy of vector fields on oscillator group}
We calculate explicitly the energy of a vector field $V\in\g$ of on the oscillator group. This
gives us the opportunity to determine some critical values of the energy functional on the oscillator group. We shall first discuss geometric properties of the map $V$ defined by a
vector field  $V\in \g$.
\begin{prop}\label{E1}
Let $G$ be the oscillator group, $V=ae_1+be_2+ce_3+de_4\in \su(2)$ be a left-invariant vector field on the oscillator group for some real constants $a,b,c,d$. Denote by $\D$ a relatively compact domain of $G$ and by $E_{\D}(V)$ the energy of $V|_{\D}$. The energy
of $V$ is; 
\begin{center}
 $E_{\D}(V)=(2+\frac{1}{4}(b+d)^2)vol(\D)$.  
\end{center}
\end{prop}
\begin{proof}
Let $G$ be the oscillator group. Consider a local orthonormal
basis $\lbrace e_1,e_2,e_3,e_4\rbrace$ of vector fields. Then, locally,
\begin{center}
$||\nabla V||^2=\sum_{i=1}^{n}\varepsilon_i g(\nabla_{e_i}V,\nabla_{e_i}V),$
\end{center}
with $\varepsilon_i=g(e_i,e_i)=\pm1$ for all indices i. Let $V \in \g$ be a left-invariant vector field on the oscillator group, then \eqref{con2} easily yields
\begin{center}
$||\nabla V||^2=\frac{1}{2}(b+d)^2.$
\end{center}
Therefore, $||\nabla V|| = 0$ if and only if $b =- d$. Thus, among vector fields of the same length, the ones with $b =- d$ will
minimize the energy.
\end{proof}
 We already know from Theorem \ref{hor4} which vector fields in $\g$ on the oscillator group
are critical points for
the energy functional. Taking into account Proposition \eqref{E1}, we then have the following. 
\begin{theorem}
Let $G$ be the oscillator group, then 
 $2vol(\D)$ is the absolute minimum value of the energy functional $E_{\D}$. Such a minimum is attained by all vector fields $V=ae_2+bu+ce_3\in \g$.
\end{theorem}
{\bf Acknowledgements}
The author wishes to express his sincere gratitude toward the professor Anna Fino for her valuable
remarks and comments.


\begin{thebibliography}{99}
%
\bibitem{a2} M.T.K. Abbassi, G. Calvaruso, D. Perrone, {\em Harmonicity of unit vector fields with respect to Riemannian g-natural metrics}, Differential Geom. Appl, 2009, 27 (1), 157-169.
\bibitem{Ba} W. Batat, K. Onda,  {\em Algebraic Ricci Solitons of four-dimensional pseudo-
Riemannian generalized symmetric spaces}, Results. Math. 64 (2013), 253-267.
\bibitem{BA} W. Batat, K. Onda,  {\em Algebraic Ricci Solitons of three-dimensional Lorentzian
Lie groups}, arXiv:1112.2455v2.
\bibitem{BA1} W. Batat, G. Calvaruso and B. De Leo,  {\em Curvature properties of Lorentzian manifolds with large isometry groups}, Math. Phys. Anal. Geom. 12 (2009), 201-217.
\bibitem{be1} L. Besse,  {\em Einstein Manifolds}, Springer Verlag, Berlin 1987.
\bibitem{MB} M. Brozos-vazquez, E. Garcia-Rio, S. Gavino-Fernandez, {\em Locally conformally flat Lorentzian gradient Ricci solitons}, J. Geom. Anal, 2013, 23, 1196-1212.
\bibitem{Ca1}
G. Calvaruso, {\em Einstein-like metrics on three-dimensional homogeneous Lorentzian manifolds}, Geom.
Dedicata, 127 (2007), 99-119.
\bibitem{c1} G. Calvaruso, {\em Harmonicity properties of invariant vector fields on three-dimensional Lorentzian Lie groups}, J. Geom. Phys. 61 (2011), 498-515.
\bibitem{ch1} M. Chaichi, E. Garcia-Rio,  Y. Matsushita, {\em Curvature properties of four-dimensional Walker metrics, Classical Quantum Gravity}, 2005, 22(3), 559-577.
\bibitem{ch2} M. Chaichi, E. Garcia-Rio,  Y. Matsushita, {\em Three-dimensional Lorentz manifolds admitting a parallel null vector field}, J. Phys. A, 2005, 38(4), 841-850.
\bibitem{Cn1} R. Cordero-Soto, E. Suazo and S. K. Suslov,  {\em Quantum integrals of motion for variable quadratic Hamiltonians}, Ann. Phys. 325 (2010), 1884-1912.
\bibitem{Cn} S. Console, G. P. Ovando, and M. Subils,   {\em Solvable models for Kodaira surfaces}, preprint arXiv:1111.2417v1 (2011).
\bibitem{DH} A. Dancer, S. Hall, M. Wang,  {\em Cohomogeneity One Shrinking Ricci Solitons: An Analytic and Numerical Study}, Asian J. Math., in press (arXiv).
\bibitem{d11} J.E. D’ Atri, H.K. Nickerson,  {\em Divergence preserving geodesic symmetries}, J. Differential Geom. 3 (1969) 467-476.
\bibitem{Dn} V. V. Dodonov, I. A. Malkin and V. I. Manko,  {\em Integrals of motion, Green functions, and coherent states of dynamical systems}, Int. J. Theor. Phys. 14 (1975), 37-54.
\bibitem{d1} S. Dragomir and D. Perrone, {\em Harmonic Vector Fields}: Variational Principles and
Differential Geometry, Elsevier, Science Ltd, 2011.
\bibitem{g11}  O. Gil-Medrano, {\em Unit vector fields that are critical points of the volume and of the energy: characterization and
examples, In: Complex, Contact and Symmetric Manifolds}, Progr. Math. 234, Birkhäuser, Boston, 2005, 165-186.
 \bibitem{g3} O. Gil-Medrano, {\em Relationship between volume and energy of vector fields}, Diff. Geom.
Appl. 15 (2001), 137-152.
\bibitem{g2} O. Gil-Medrano, A. Hurtado, {\em Spacelike energy of timelike unit vector fields on a Lorentzian manifold}, J. Geom. Phys., 2004, 51(1), 82-100.
\bibitem{i} T. Ishihara, {\em Harmonic sections of tangent bundles}, J. Math. Tokushima Univ., 1979, 13, 23-27.
\bibitem{h1}  D. S. Han,  J. W. Yim,  {\em Unit vector fields on spheres, which are harmonic maps}, Math. Z. 1998, 227(1), 83-92.
\bibitem{Lm} R. Lafunte, J. Lauret,  {\em On homogeneous Ricci solitons}, arXiv:1210.3656v1
\bibitem{LM} R. Lafunte, J. Lauret,  {\em Structure of homogeneous Ricci solitons and the Alekseevskii conjectur}, J. Diff. Geom, Volume 98, Number 2 (2014), 315-347.
\bibitem{LJ} J. Lauret,  {\em Ricci soliton solvmanifolds}, J. reine angew. Math., 650 (2011), 1-21.
\bibitem{Ln} A. V. Levichev, {\em Methods of investigation of the causal structure of homogeneous Lorentz
manifolds}, Siberian Math. J. 31, 395-408 (1990).
\bibitem{Ln1} A. V. Levichev, {\em Several symmetric spaces of general relativity theory as solutions of Einstein-
Yang-Mills equations}, in Group theoretical methods in physics, Proc. 3rd Intern. Sem. Yurmala,
May 1985, vol. 1, (Nauka, Moscow, 1986), pp. 145-150 (in Russian).

\bibitem{Mn} A. Medina, {\em Groupes de Lie munis de m´etriques bi-invariantes}, Tohoku Math. J. 37, 405 (1985).
\bibitem{N1}  C. Nappi, and E. Witten,  {\em A WZW model based on a non-semi-simple group}, hep-th/
9310112, 1993.
\bibitem{n} O. Nouhaud, {\em Applications harmoniques dune variété riemannienne dans son fibré tangent}, Généralisation, C. R.
Acad. Sci. Paris Sér. A-B, 1977, 284(14), A815-A818.
\bibitem{on} K. Onda, {\em Examples of algebraic Ricci solitons in the
pseudo-Riemannian case}, Acta Math. Hungar, 2014, online.
 \bibitem{p1} D. Perrone, {\em Harmonic characteristic vector fields on contact metric three-manifolds}, Bull.
Austral. Math. Soc. 67 (2003), 305-315.
\bibitem{p2} D. Perrone, {\em Contact metric manifolds whose characteristc vector field is a harmonic vector
field}, Diff. Geom. Appl. 20 (2004), 367-378.
\bibitem{p112} F. Podesta, A. Spiro,  {\em Four-dimensional Einstein-like manifolds and curvature homogeneity}, Geom. Dedicata 54 (1995) 225-243.
\bibitem{s1} Z.I. Szabó,  {\em Spectral theory for operator families on Riemannian manifolds}, Proc. Symp. Pure Math. 54 (3) (1993) 615-665.
\bibitem{Wa1} A.G. Walker, {\em On parallel fields of partially null vector spaces}, Quart. J. Math.
 Oxford  \textbf{20} (1949), 135-145.
\bibitem{Wa2} A.G. Walker, {\em Canonical form for a Riemannian space with a parallel field of null planes}, Quart. J. Math. Oxford (2) \textbf{1} (1950), 69-79.
\bibitem{Z1}  A. Zhegib, {\em Sur les espaces-temps homogenes, Prepublication}, Ecole Normale, Lyon, 1995.
\end{thebibliography}
\end{document}